\documentclass[11pt]{article}

\usepackage{amsmath, amssymb, amsthm, amstext, graphicx,tikz}
\usepackage{enumerate}
\usepackage{fullpage}

\newtheorem{theorem}{Theorem}[section]
\newtheorem{lemma}[theorem]{Lemma}

\newtheorem{proposition}[theorem]{Proposition}

\begin{document}

\title{Distinct volume subsets}
\author{{David Conlon}\thanks{Mathematical Institute, Oxford OX2 6GG, UK. Email: {\tt david.conlon@maths.ox.ac.uk}. Research supported by a Royal Society University Research Fellowship.}
\and{Jacob Fox}\thanks{Department of Mathematics, MIT, Cambridge, MA 02139-4307. Email: {\tt fox@math.mit.edu}. Research supported by a Packard Fellowship, by a Simons Fellowship, by NSF grant DMS-1069197, by an Alfred P. Sloan Research Fellowship and by an MIT NEC Corporation Award.}
\and
{William Gasarch}\thanks{
Department of Computer Science,
       University of Maryland at College Park, College Park, MD 20742.
Email: \texttt{gasarch@cs.umd.edu}.
}
\and
{David G. Harris}\thanks{Department of Applied Mathematics,
        University of Maryland at College Park, College Park, MD 20742.
Email: \texttt{davidgharris29@hotmail.com}.
}
\and
{Douglas Ulrich}\thanks{
Department of Mathematics,
        University of Maryland at College Park, College Park, MD 20742.
Email: \texttt{ds\_ulrich@hotmail.com}.
}
\and
{Samuel Zbarsky}\thanks{
Department of Mathematics,
Carnegie Mellon University, Pittsburgh, PA 15213.
Email: \texttt{sa\_zbarsky@yahoo.com}.
}
}
\date{}

\maketitle

\begin{abstract}
Suppose that $a$ and $d$ are positive integers with $a \geq 2$. Let $h_{a,d}(n)$ be the largest integer $t$ such that any set of $n$ points in $\mathbb{R}^d$ contains a subset of $t$ points for which all the non-zero volumes of the ${t \choose a}$ subsets of order $a$ are distinct. Beginning with Erd\H{o}s in 1957, the function $h_{2,d}(n)$ has been closely studied and is known to be at least a power of $n$. We improve the best known bound for $h_{2,d}(n)$ and show that $h_{a,d}(n)$ is at least a power of $n$ for all $a$ and $d$.
\end{abstract}

\section{Introduction} 

The Erd\H{o}s distinct distance problem is one of the foundational problems in discrete geometry. This problem, first stated by Erd\H{o}s~\cite{Er46} in 1946, asks for an estimate of the minimal number $g_d(n)$ of distances defined by $n$ points in $\mathbb{R}^d$. In the plane, the $\sqrt{n} \times \sqrt{n}$ grid shows that $g_2(n)=O\left(n/\sqrt{\log n}\right)$ and Erd\H{o}s conjectured that $g_2(n)=n^{1-o(1)}$. This conjecture was recently confirmed by Guth and Katz \cite{GK13}, who proved that $g_2(n)=\Omega\left(n/\log n\right)$. For $d \geq 3$, a $d$-dimensional grid demonstrates that $g_d(n)=O_d(n^{2/d})$ and this bound is also believed to be optimal.\footnote{Here, $O_d$ is the usual big-$O$ notation but the subscript indicates that the constant may depend on $d$. We use a similar convention for $\Omega_d$ and $\Theta_d$.} In high dimensions, an almost matching lower bound was given by Solymosi and Vu~\cite{SV}, who showed that $g_d(n) = \Omega_d(n^{(2 - \epsilon)/d})$, where $\epsilon$ tends to $0$ as $d$ tends to infinity.

Suppose that $2 \leq a \leq d+1$. A well-studied generalization of the distinct distance problem (see, for example, \cite{EP71, EP95, EPS}) asks for the minimal number $g_{a,d}(n)$ of nonzero $(a-1)$-dimensional volumes determined by the subsets of order $a$ of $n$ points in $\mathbb{R}^d$, assuming that not all points lie on a common $(a-2)$-dimensional hyperplane. Note that $g_d(n) = g_{2, d} (n)$. For $a = d+1$, it is easy to see that $g_{d+1,d}(n) \leq \lfloor \frac{n-1}{d}\rfloor$ by taking $d$ sets of about $n/d$ equally spaced points on parallel lines through the vertices of a $(d-1)$-simplex. 
Erd\H{o}s, Purdy and Straus \cite{EPS} conjectured that this bound is tight for $n$ sufficiently large  depending on $d$. Following progress in \cite{BP} and \cite{DT07}, Pinchasi~\cite{P08} finally solved the $d=2$ case of this conjecture by proving that $g_{3,2}(n)=\lfloor \frac{n-1}{2}\rfloor$. Dumitrescu and Cs.~T\'oth~\cite{DuTo} proved that $g_{d+1,d}(n)=\Theta_d(n)$ for all $d$.

A related problem of Erd\H{o}s~\cite{Er57, Er70} from 1957 asks for $h_d(n)$, the maximum $t$ such that every $n$-point set $P$ in $\mathbb{R}^d$ contains a subset $S$ of $t$ points such that all ${t \choose 2}$ distances between the pairs of points in $S$ are distinct. Erd\H{o}s conjectured that $h_1(n)=(1+o(1))\sqrt{n}$. The set $P=\{1,\ldots,n\}$ gives the upper bound $h_1(n) \leq (1+o(1))\sqrt{n}$, while a lower bound of the form $h_1(n) = \Omega(\sqrt{n})$ follows from a result of Koml\'os, Sulyok and Szemer\'edi \cite{KSS73}. In two dimensions, building on earlier results in \cite{AEP} and \cite{LT} and utilizing an important estimate from the work of Guth and Katz \cite{GK13}, Charalambides \cite{Ch} improved the bound\footnote{Actually, the bound stated in that paper is slightly worse, but a  careful analysis of the proof gives this bound.} to $h_{2}(n)=\Omega(n^{1/3}/\log^{1/3} n)$. Since the $\sqrt{n} \times \sqrt{n}$ grid has $O(n/\sqrt{\log n})$ distinct distances, it follows that $h_2(n)=O(n^{1/2}/\log^{1/4} n)$. For all $d \geq 3$, Thiele~\cite{Th} showed that $h_{d}(n)=\Omega_d (n^{1/(3d-2)})$. We give the following improvement to this bound. 

\begin{proposition} \label{prop:distances}
For each integer $d \geq 2$, there exists a positive constant $c_d$ such that
$$h_{d}(n) \geq c_dn^{\frac{1}{3d-3}}\left(\log n\right)^{\frac{1}{3}-\frac{2}{3d-3}}.$$
\end{proposition}

\noindent
An upper bound comes from the $d$-dimensional grid with sides of length $n^{1/d}$. This set contains $n$ points and the total number of distances between pairs is $O_d(n^{2/d})$. It follows that a distinct distance subset must have $O_d(n^{1/d})$ points. Hence, $h_{d}(n) =O_d(n^{1/d})$. 

We are interested in estimating the following generalization of this function. Let $h_{a,d}(n)$ be the largest integer $t$ such that any set of $n$ points in $\mathbb{R}^d$ contains a subset of $t$ points for which all the non-zero volumes of the ${t \choose a}$ subsets of order $a$ are distinct. In particular, $h_d(n)=h_{2,d}(n)$. Note that if we did not disregard subsets of volume zero, we could place all of our points on a hyperplane of dimension $a - 2$. Then, since every subset of order $a$ has volume zero, the largest distinct volume subset would have $a$ points. We compare only non-zero volumes so as to avoid this degeneracy (though an alternative solution is discussed in the concluding remarks).

 For infinite sets, an analogous function was studied by Erd\H{o}s \cite{Er50} in the 1950s. However, to the best of our knowledge, the finite version has not been studied before. Our main theorem says that $h_{a,d}(n)$ is of polynomial size for all $a$ and $d$. This is trivial for $a > d + 1$, since then all volumes are necessarily zero. For $a \leq d+1$, our main result is as follows.

\begin{theorem}\label{main1}
For all integers $a$ and $d$ with $2 \leq a \leq d+1$, there exists a positive constant $c_{a,d}$ such that
\[h_{a,d} (n) \geq c_{a,d} n^{\frac{1}{(2a-1)d}}.\]
\end{theorem}

\noindent
In certain special cases, this bound can be significantly improved. In particular, when $a = d +1$, we will show that $h_{d+1, d}(n) \geq c_d n^{1/(2d+2)}$.

An upper bound again follows from considering the $n^{1/d} \times \dots \times n^{1/d}$ grid. For example, since the square of the area formed by any three points in $\mathbb{R}^d$ may be written as a quartic polynomial with integer coefficients, the triples of $n^{1/d} \times \dots \times n^{1/d}$ form $O_d(n^{4/d})$ distinct areas. On the other hand, suppose that $X$ is a distinct area subset of $n^{1/d} \times \dots \times n^{1/d}$. Then, since the line through any two points contains at most $n^{1/d}$ other points, the points of $X$ form at least $|X|(|X|-1)(|X| - n^{1/d})/6$ distinct non-zero areas. Comparing this with the upper bound, we see that $h_{3,d}(n) = O_d(n^{\frac{4}{3 d}})$. For $a \geq 4$, a slight variant of this analysis shows that $h_{a,d} = O_{a,d} (n^{\frac{a-2}{d}})$.

We will estimate the largest distinct volume subset in a set of $n$ points by considering a coloring of the complete $a$-uniform hypergraph on the $n$ points, where we color an $a$-set depending on the volume of the simplex formed by the $a$-set. Our aim then is to find a large rainbow clique in this coloring, that is, a subset of points   whose edges have different colors. The key lemma, proved in Section~\ref{rainbow-sec}, says that if the coloring is sparse in each color, then it contains large rainbow cliques. As we will see in Section~\ref{spec-sec}, this lemma easily allows us to prove lower bounds on $h_{2,d}(n)$ and $h_{d+1, d}(n)$. In order to extend the method to give lower bounds on $h_{a,d}(n)$ for all $2 < a < d+1$, we will need to use some tools from algebraic geometry. This will be discussed in Section~\ref{general-sec}. We then conclude by discussing alternative definitions, infinite versions, algorithmic aspects and open problems.

\section{Rainbow cliques in $m$-good complete hypergraph colorings}
\label{rainbow-sec}
Call an edge-coloring of a $k$-uniform hypergraph {\it $m$-good} if each $(k-1)$-tuple of vertices is contained in at most $m$ edges of any particular color. In particular, in a $1$-good coloring, the edges containing any given $(k-1)$-tuple are all different colors. 

Let $g_k(m,t)$ denote the smallest $n$ such that every $m$-good edge-coloring
of the complete $k$-uniform hypergraph $K_n^{(k)}$ on $n$ vertices yields a rainbow copy of $K_t^{(k)}$.  Alon,  Jiang, Miller and Pritikin \cite{AJMP} proved that $g_2(m,t)=\Theta(mt^3/\log t)$. Here we prove a general estimate on $g_k(m,t)$. 

\begin{lemma}\label{keylem}
For positive integers $k, m$ and $t$ with $k \geq 2$, $g_k(m,t) \leq 4 mt^{2k-1}$. 
\end{lemma}
\begin{proof}
Consider an $m$-good edge-coloring $c$ of $K_n^{(k)}$ with $n= 4m t^{2k-1}$. We wish to prove that this coloring contains a rainbow $K_t^{(k)}$. If $t \leq k$, every $t$-set is trivially rainbow, so we may assume $t > k$. 

We first give an upper bound on $A_s$, the number of unordered pairs $\{e_1,e_2\}$ of distinct edges with $c(e_1)=c(e_2)$ and $|e_1 \cap e_2|=s$. For a given color $\gamma$ and an $s$-set $S$ of vertices, let $B_{\gamma} (S)$ be the number of edges $e$ with  
$c(e)=\gamma$ and $S \subset e$. As each superset of $S$ of size $k-1$ is contained in at most $m$ sets of size $k$ and color $\gamma$, we have $B_{\gamma} (S) \leq B_s:=\frac{m}{k-s}{n-s \choose k-1-s}$, where we divide by $k-s$ to account for the fact that we count any particular edge $k-s$ times. By adding over all ${n \choose s}$ possible $s$-sets $S$ and noting that $\sum_{\gamma} \binom{B_\gamma(S)}{2}$ is maximized (given that $\sum_{\gamma} B_{\gamma}(S)={n-s \choose k-s}$)  when all of the $B_{\gamma}(S)$ are as large as possible, it follows that
\begin{eqnarray*}
A_s & \leq & \sum_S \sum_{\gamma} \binom{B_\gamma(S)}{2} \leq {n \choose s}\frac{{n-s \choose k-s}}{B_s}{B_s \choose 2} = {n \choose s}\frac{n-k+1}{m}{B_s \choose 2}\\ 
& \leq & \frac{n^{s+1}}{2m \cdot s!}\left(\frac{m}{k-s}{n-s \choose k-1-s}\right)^2 \leq \frac{m n^{2k-s-1}}{2s!(k-s)!^2}.
\end{eqnarray*}
Let $T$ be a random subset of order $2t$. The expected number of unordered pairs $\{e_1,e_2\}$ of distinct edges in $T$ with $c(e_1)=c(e_2)$ is, by summing over all possible intersection sizes of $e_1 \cap e_2$,  
\begin{eqnarray*}\sum_{s=0}^{k-1} A_s{2t \choose 2k-s}/{n \choose 2k-s} & \leq & \sum_{s=0}^{k-1} A_s\left(\frac{2t}{n}\right)^{2k-s} \leq \sum_{s=0}^{k-1} \frac{m n^{2k-s-1}}{2s!(k-s)!^2} \left(\frac{2t}{n}\right)^{2k-s}  \\ & = & 
\frac{mt^{2k}}{n}\sum_{s=0}^{k-1} \frac{2^{2k}}{2s!(k-s)!^2}(2t)^{-s} \leq \frac{4 mt^{2k}}{n} =  t.\end{eqnarray*}
Hence, there is a set of order $2t$ with at most $t$ pairs of distinct edges with the same color. Deleting one vertex from each such pair of edges, there is a subset of order at least $t$ in which all edges have different colors. Hence, the coloring contains a rainbow $K_t^{(k)}$, completing the proof.
\end{proof}

Following the method of~\cite{AJMP}, this result can be improved to $g_k(m,t) = O_k(m t^{2k-1}/\log t)$. However, since our final results are unlikely to be sharp in their polynomial dependence, we will not track these additional logarithmic factors. 

\section{First estimates} 
\label{spec-sec}
Let $s_d(t)$ denote the minimum $n$ such that every set of $n$ points on the $d$-dimensional sphere $\mathbb{S}^d : = \{x \in \mathbb{R}^{d+1} : \|x\| = 1\}$ contains a subset of order $t$ such that all ${t \choose 2}$ distances between pairs of points are distinct. The work of Charalambides \cite{Ch} gives the bound $s_2(t)=O(t^3\log t)$. We use this as the base case in the following lemma.

\begin{lemma}
For all integers $d, t \geq 2$, $$s_d(t) \leq  g_2(s_{d-1}(t),t) =O(s_{d-1}(t)t^3/\log t).$$
In particular, there exists a positive constant $C_d$ such that $$s_d(t) \leq C_d t^{3d-3}(\log t)^{3-d}.$$
\end{lemma}
\begin{proof}
Suppose we have a set $P$ of $n=g_2(s_{d-1}(t),t)$ points in $\mathbb{S}^d$. If there is a point $p \in P$ which has $m=s_{d-1}(t)$ points equidistant from $p$, then these $s_{d-1}(t)$ points lie on a sphere of dimension $d-1$. From the definition of $s_{d-1}(t)$, this would imply that there is a subset of $t$ points such that all ${t \choose 2}$ distances between pairs of points are distinct and we would be done. Hence, we may assume otherwise. Color the edge between each pair of points of $P$ by the distance between them. Since this coloring is $m$-good, the definition of $g_2(m,t)$ implies that there must be a rainbow $K_t$ in this edge-coloring of $K_n$. The vertex set of this rainbow $K_t$ is the desired set of $t$ points with distinct distances between each pair of points. By the result of Alon et al.~\cite{AJMP}, we have $g_2(m,t)=O(mt^3/\log t)$, which, since $m = s_{d-1}(t)$, completes the proof of our recursive estimate. The bound $s_d(t) \leq C_d t^{3d-3}(\log t)^{3-d}$ follows from this recursive estimate and the case $d=2$.  
\end{proof}

Let $H_{a,d}(t)$ be the inverse function of $h_{a,d}(n)$. More precisely, $H_{a,d}(t)$ is the minimum $n$ such that any set of $n$ points in $\mathbb{R}^d$ contains a subset of $t$ points for which all the non-zero volumes of the ${t \choose a}$ subsets of size $a$ are distinct. Essentially the same proof as above gives the following result, which implies Proposition~\ref{prop:distances}.   

\begin{proposition}
For all integers $d, t \geq 2$, $$H_{2,d}(t) \leq  g_2(s_{d-1}(t),t) =O(s_{d-1}(t)t^3/\log t).$$
In particular, there exists a positive constant $c_d$ such that $$h_{2,d}(n) \geq c_dn^{\frac{1}{3d-3}}(\log n)^{\frac{1}{3}-\frac{2}{3d-3}}.$$
\end{proposition}

Another straightforward case is when $a = d+1$ because we may exploit the fact that the locus of points forming a given volume with a fixed set of $d$ points is a pair of parallel hyperplanes.

\begin{proposition} \label{prop:simplex}
For all integers $d \geq 2$ and $t \geq d + 1$, 
$$H_{d+1, d}(t) \leq g_{d+1}(2t,t) \leq 8 t^{2d+2}.$$
In particular,
$$h_{d+1,d}(n) \geq n^{\frac{1}{2 d + 2}}/2.$$
\end{proposition} 
\begin{proof}
Consider a set $P$ of $n = g_{d+1}(2t,t)$ points in $\mathbb{R}^d$. For a given subset $D$ of size $d$ and a given number $\ell>0$, the locus $L$ of points in $\mathbb{R}^d$ which together with $D$ have volume $\ell$ forms two hyperplanes parallel to and on opposite sides of the hyperplane containing $D$. If either of these hyperplanes contains $t$ points of $P$, we have found the required subset, with every volume being zero (recall that our definition only required distinct non-zero volumes). We may therefore assume that there are at most $2t$ points on $L$. Consider a coloring of the edges of the complete $(d+1)$-uniform hypergraph with vertex set $P$, where each edge with zero volume receives a unique color and each edge of non-zero volume is colored by that volume. By the above discussion, this edge-coloring is $2t$-good. From the definition of $g_{d+1}(2t,t)$, the set $P$ must contain a rainbow clique of order $t$, which is a set of $t$ points such that all non-zero volumes of subsets of order $d+1$ are different. Lemma \ref{keylem} implies that  $g_{d+1}(2t,t) \leq 8t^{2d+2}$ and the lower bound on $h_{d+1,d}(n)$ follows. 
\end{proof}

\section{The general case}
\label{general-sec}
To prove our lower bound on $h_{a,d}$, we need to prove a more general
theorem. 
This will require concepts and results from algebraic geometry. 
Rather than working in $\mathbb{R}^N$, it will be useful to work in a projective space over an algebraically closed field. We will accordingly consider our set of points as a subset of $\mathbb{CP}^N$, projective $N$-space over the complex numbers. This space corresponds to the set of complex lines in $\mathbb{C}^{N+1}$, that is, two points are identified if one is a complex multiple of the other. 

In $\mathbb{R}^N$, notions of area and volume (or, more correctly, their squares) can be defined
in terms of multivariate polynomials. When we are working over projective space, we cannot similarly define the volume of a set of points. However, given a polynomial defined on $\mathbb{R}^N$, we can lift such an expression to a homogeneous polynomial defined on $\mathbb{CP}^N$. This is necessary for any equations defined by these polynomials to be well-defined over projective space. To give an example, consider the equation which says that the distance $d(y, u)$ between two points $y, u \in \mathbb{R}^N$ is equal to $\ell$. If $y = (y_1, \dots, y_N)$ and $u = (u_1, \dots, u_N)$, then this equation may be written as
\[\sum_{i=1}^N (y_i - u_i)^2 = \ell^2.\]
If we consider $u$ as fixed and $y$ as a variable, this is an equation in $N$ real variables $y_1, \dots, y_N$. This may now be lifted to a homogeneous polynomial on $\mathbb{CP}^N$ by setting $y_i = \frac{x_i}{x_0}$ for each $i$ and multiplying out. This yields
\[\sum_{i=1}^N (x_i - u_i x_0)^2 = \ell^2 x_0^2,\]
which is a homogeneous polynomial in the variables $x_0, x_1, \dots, x_N$, where we now allow these variables to take complex values. Note that if $(x_0, x_1, \dots, x_N)$ satisfies this equation, then so does $(\lambda x_0, \lambda x_1, \dots, \lambda x_N)$ for any $\lambda \in \mathbb{C}$. That is, the equation is well-defined on $\mathbb{CP}^N$. Moreover, any  $y$ which satisfied the original equation still satisfies this new equation since, in $\mathbb{CP}^N$, the point $y \in \mathbb{R}^N$ corresponds to the point $(1, y)$ and its multiples.

Homogeneous polynomials are of fundamental importance in algebraic geometry. Indeed, the basic object of study in this field is the {\it variety}, defined to be the set of solutions in $\mathbb{CP}^N$ to a collection of homogeneous polynomials $f_1(x) = \dots = f_k(x) = 0$. 
We say that $V$ is \emph{irreducible} if it cannot be written as $V = C_1 \cup C_2$, where $C_1$ and $C_2$ are distinct, non-empty closed sets in $\mathbb{CP}^N$, neither of which equals $V$.
To every variety, there is an invariant referred to as the \emph{Hilbert polynomial}. 
The \emph{dimension} of $V$ is the degree $d$ of the Hilbert polynomial of $V$ and its \emph{degree} is defined to be $d!$ times the leading coefficient of the Hilbert polynomial of $V$. We will not need to know the explicit definition of the Hilbert polynomial in what follows. However, some intuition may be gained by noting that if the polynomials $f_1, \dots, f_k$ are well-behaved, then the dimension is $d = N-k$ and the degree of $V$ is the product of the degrees of $f_1, \dots, f_k$.

For our purposes, we only need one key result from algebraic geometry. This follows from Theorem I, 7.7 of Hartshorne's book on algebraic geometry~\cite{Ha}.

\begin{lemma}\label{le:hbt}
Let $V$ be an irreducible variety of dimension $d$ and $f$ a homogeneous polynomial.
Let 
$$W=V\cap \{x \mid f(x)=0\}.$$
Either $W=V$ or all of the following must hold:
\begin{enumerate}
\item $W$ is the union of irreducible varieties $W = Z_1 \cup \dots \cup Z_j$.
\item The degrees of $Z_1, \dots, Z_j$ are bounded by a function of the degree of $V$ and the degree of $f$.
\item The number of components $j$ is bounded by a function of the degree of $V$ and the degree of $f$.
\item All of the components $Z_1, \dots, Z_j$ have dimension exactly $d-1$.
\end{enumerate}
\end{lemma}

\noindent
Note that if $d = 1$, this gives a form of B\'{e}zout's theorem: the intersection $W$ consists of components of dimension 0 and bounded degree, that is, a bounded number of isolated points.

In order to facilitate our induction, it will be useful to consider a function which is more general than $h_{a,d}(n)$ and allows for the points to be taken within a $d$-dimensional irreducible variety of degree $r$. In what follows, we will be concerned with two different notions of dimension, the dimension $d$ of the variety and the dimension $N$ of the space in which it is embedded. When we mention a variety of dimension $d$, we will always assume that it is in $\mathbb{CP}^N$ for some $N \geq d$.
 
Let $H_{a,d,r}(t)$ be the minimum $n$ such that any set of $n$ points in $V \cap \mathbb{R}^N$, where $V$ is an irreducible variety of dimension $d$ and degree $r$, contains a subset of $t$ points for which all the non-zero volumes of the ${t \choose a}$ subsets of size $a$ are distinct. In the degenerate case $d = 0$, we define $H_{a, 0, r}(t) = 1$. Our main theorem is now as follows.

\begin{theorem}\label{thm:general}
For all integers $r, d \geq 1$ and $a \geq 2$, there exist positive integers $r'$ and $j$ such that, for all integers $t \geq a$,  
$$H_{a,d, r}(t) \leq g_a(j H_{a, d-1, r'}(t), t) \leq 4 j H_{a, d-1, r'}(t) t^{2a - 1}.$$ 
In particular, there exists a positive constant $c_{a,d}$ such that
$$h_{a,d}(n) \geq c_{a,d} n^{\frac{1}{(2a - 1)d}}.$$
\end{theorem}
 
\begin{proof}
Consider a set $P$ of $n$ points in $V \cap \mathbb{R}^N$, where $V$ is an irreducible variety of dimension $d$ and degree $r$. For a given subset $A$ of $P$ of size $a - 1$ and a given $\ell > 0$, consider the set of points $x$ in $\mathbb{R}^N$ which together with $A$ have volume $\ell$. By our earlier discussions, it is possible to lift the equation saying that the volume of the simplex formed by $A \cup \{x\}$ is equal to $\ell$ to a homogeneous equation $f(x) = 0$. 

Consider now the set $W = V\cap \{x \mid f(x)=0\}$. By Lemma~\ref{le:hbt}, either $W = V$ or $W$ splits into a bounded number of components with dimension $d-1$ and bounded degree. If $W = V$, this implies that the volume of the simplex formed by $A \cup \{x\}$ is equal to $\ell$ for all points $x$ in $V \cap \mathbb{R}^N$. Taking $x \in A$, we see that this volume must be zero, contradicting our assumption that $\ell > 0$.

We may therefore assume that $W = Z_1 \cup \dots \cup Z_j$, where each $Z_i$ is an irreducible variety of degree at most $r'$ and both $r'$ and $j$ depend only on $a$, $d$ and $r$. Suppose that $d \geq 2$. If any $Z_i$ contains $H_{a, d-1, r'} (t)$ points of $P$, then, by definition, $P$ contains a subset of $t$ points for which all the non-zero volumes of the ${t \choose a}$ subsets of size $a$ are distinct. Therefore, $W$ contains at most $j H_{a, d - 1, r'}(t)$ points of $P$. In the $d = 1$ case, each $Z_i$ is an isolated point. Since $H_{a, 0, r'}(t) = 1$, this again implies that $W$ contains at most $j H_{a, d-1, r'}(t)$ points of $P$.
 
Now consider a coloring of the complete $a$-uniform hypergraph with vertex set $P$, where each edge with zero volume receives a unique color and each edge of non-zero volume is colored by that volume. By the above discussion, we know that this coloring is $j H_{a, d-1, r'}(t)$-good. For $n =  g_a(j H_{a, d-1, r'}(t), t)$, the definition of $g_a$ implies that there must be a rainbow clique of order $t$, that is, a set of $t$ points such that all non-zero volumes of the subsets of order $a$ are distinct. 

Iterating the recurrence relation $H_{a,d, r}(t) \leq 4 j H_{a, d-1, r'}(t) t^{2a - 1}$ easily implies that $H_{a,d,r}(t) \leq C_{a,d,r} t^{(2a-1)d}$ for some constant $C_{a,d,r}$. The lower bound on $h_{a,d}(n)$ now follows from noting that $\mathbb{R}^d \subset \mathbb{CP}^d$, which is a variety of dimension $d$ and degree $1$. 
\end{proof}
 
\section{Concluding remarks}

\subsection{Alternative definitions}

If one wishes to find a large distinct volume subset containing only non-zero volumes, it is necessary to make some additional assumption about the set of points we are considering. The most natural assumption is to suppose that there are no zero volumes in the set, that is, that no $a$ points lie on an $(a-2)$-dimensional subspace. The function we are then interested in is $h'_{a,d}(n)$, defined to be the largest integer $t$ such that any set of $n$ points in $\mathbb{R}^d$, no $a$ on the same $(a-2)$-dimensional space, contains a subset of $t$ points for which all the volumes of the ${t \choose a}$ subsets of size $a$ are distinct. Note that unlike $h_{a,d}$ this definition only makes sense when $a \leq d+1$. Nevertheless, for $2 \leq a \leq d+1$, it is not hard to alter the proof of Theorem~\ref{main1} to show that
\[h'_{a,d}(n) \geq c_{a,d} n^{\frac{1}{(2a-1)d}}.\]
Moreover, in the particular case where $a = d+1$, we may show that $h'_{d+1,d}(n) \geq c_d n^{1/(2d+1)}$, improving slightly on the bound for $h_{d+1,d}(n)$. The proof of this is almost identical to Proposition~\ref{prop:simplex}, but uses the fact that there are at most $d$ points on any hyperplane.

As the $d$-dimensional grid contains many collinear points, the upper bounds for $h_{a,d}(n)$ discussed in the introduction are not valid under this alternative definition. However, we can instead consider a subset of the $d$-dimensional grid with no $a$ points lying on an $(a-2)$-dimensional hyperplane. In the case $a=3$, finding such a set is known as the no-three-in-line problem and is well-studied, dating back to 1917~\cite{1917}. For fixed positive integers $d$ and $a$, let $F_{a,d}(t)$ be the maximum number of points in the $d$-dimensional grid $t \times \cdots \times t$ such that no $a$ of these points are contained in an $(a-2)$-dimensional affine subspace. For fixed $d$, it is shown in~\cite{L13} that $F_{3,d}(t) = \Omega_d(t^{d-2})$. This may in turn be used to show that $h'_{3,d}(n) = O_d(n^{4/3(d-2)})$. Further estimates on $F_{a,d}(t)$ which imply upper bounds on $h'_{a,d}(n)$ are given in \cite{L13}. 

\subsection{Infinite sets}

As mentioned in the introduction, an infinite variant of our problem was first studied by Erd\H{o}s~\cite{Er50} in 1950. Under the assumption of the axiom of choice, he proved that $h_{2,d}(\alpha) = \alpha$ for all $\alpha \leq 2^{\aleph_0}$. Moreover, for any regular cardinal $\alpha \leq 2^{\aleph_0}$ and any $2 \leq a \leq d+1$, he proved that $h_{a,d}(\alpha) = \alpha$. This result may fail for singular cardinals $\alpha$. 


\subsection{Algorithmic aspects}

We have proved that for all $2 \leq a \leq d+1$, any set of $n$ points in $\mathbb{R}^d$ contains a subset of $t = c_{a,d} n^{\frac{1}{(2a-1)d}}$ points such that all non-zero volumes formed by the $\binom{t}{a}$ sets of order $a$ are distinct. If one tracks the proof carefully, it is possible to show that our proof gives a randomized algorithm with running time $O_d(n^{d + O(1)})$. The central role here is played by Lemma~\ref{keylem}. Indeed, given $n \geq 4 m t^{2 a-1}$ and a coloring of the complete $a$-uniform hypergraph on $n$ vertices, then, in running time $n^{a + O(1)}$, we can find either a set of $a-1$ vertices with more than $m$ neighbors (it takes time at most $n^a$ to check these) or, if there is no such set, a rainbow set of order $t$. Following Lemma~\ref{keylem}, this is done by sampling a random set of size $2t$ and removing a single vertex from each pair with the same color. Note that comparing all pairs takes time at most $(2t)^{2a} \leq n^2$.

However, one can do better. We say that an edge $e$ is bad if it contains $a-1$ points which are contained in more than $m$ edges of the same color as $e$. Otherwise, we say that $e$ is good. Following the proof of Lemma~\ref{keylem}, but focusing now on pairs of distinct good edges, one can show that a random set $T$ of size $2t$ will, in expectation, contain at most $t$ pairs of good edges with the same color. We now search the set $T$ for bad edges. This takes time at most $(2t)^{a-1} n \leq n^2$ and if we find a bad edge, we are done. Otherwise, we know that all edges are good, so we may remove a vertex from each pair with the same color to produce a rainbow set of order at least $t$. Again, comparing edges has a running time of at most $(2t)^{2a} \leq n^2$. That is, we have an algorithmic version of Lemma~\ref{keylem} with running time $n^{O(1)}$. This may then be used to obtain a constructive version of Theorem~\ref{main1} running in time $O_{d}(n^{O(1)})$.

\subsection{Open problems}

If one uses the improved bound $g_k(m, t) = O_k(m t^{2k-1}/\log t)$ noted in Section $2$, 
Theorem~\ref{thm:general} easily implies that 
$$h_{a,d}(n) \geq c_{a,d} n^{\frac{1}{(2a-1)d}} (\log n)^{\frac{1}{2a-1}}.$$
It would be interesting to know if the polynomial dependence in this bound can be significantly improved. This would be particularly interesting in the case $a = d = 2$, where the outstanding open problem is to determine whether $h_{2}(n) = n^{1/2 - o(1)}$. If true, a proof of this fact is likely to require new ideas. 

\vspace{3mm}

\noindent
{\bf Acknowledgements.} The authors would like to thank Tucker Bane, Andrew Lohr, Jared Marx-Kuo, Joe Mileti, Jessica Shi, Srinivas Vasudevan and Yufei Zhao for helpful discussions.

\end{document}